\documentclass[14pt]{amsart}
\usepackage[margin=3cm]{geometry}
\usepackage{stmaryrd}

\usepackage{microtype} 

\usepackage{amsthm,amsmath,amssymb,amsfonts}
\usepackage[alphabetic]{amsrefs}
\usepackage[skip=4pt plus1pt, indent=10pt]{parskip}
\usepackage{enumerate}
\usepackage{xcolor}
\usepackage{tikz}
\usepackage[colorlinks,bookmarks, anchorcolor=blue, citecolor=blue, linkcolor=blue,urlcolor=blue, bookmarksopenlevel=1]{hyperref}
\usepackage{tikz-cd}

\newtheorem{theorem}{Theorem}[section]
\newtheorem{proposition}[theorem]{Proposition}
\newtheorem{lemma}[theorem]{Lemma}
\newtheorem{corollary}[theorem]{Corollary}

\theoremstyle{definition}
\newtheorem{definition}[theorem]{Definition}
\newtheorem{remark}[theorem]{Remark}

\begin{document}\baselineskip=15pt

\begin{center}
\title[Finite generation and slope of divisors on moduli of cubics]{Finite generation of Noether--Lefschetz divisors and the slope of the moduli space of cubic fourfolds}

\author{Ignacio Barros}
\address{\parbox{0.9\textwidth}{
Department of Mathematics\\[1pt]
Universiteit Antwerpen\\[1pt]
Middelheimlaan 1, 2020 Antwerpen, Belgium
\vspace{1mm}}}
\email{{ignacio.barros@uantwerpen.be}}

\author{Shi He}
\address{\parbox{0.9\textwidth}{
Department of Mathematics\\[1pt]
Universiteit Antwerpen\\[1pt]
Middelheimlaan 1, 2020 Antwerpen, Belgium
\vspace{1mm}}}
\email{{shi.he@uantwerpen.be}}

\author{Paul Kiefer}
\address{\parbox{0.9\textwidth}{
Department of Mathematics\\[1pt]
Universiteit Antwerpen\\[1pt]
Middelheimlaan 1, 2020 Antwerpen, Belgium
\vspace{1mm}}}
\email{{paul.kiefer@uantwerpen.be}}

\thanks{I.B. was supported by the Research Foundation – Flanders (FWO) project number G0D9323N and by the Deutsche Forschungsgemeinschaft (DFG) - Project-ID 491392403 – TRR 358. P.K. was supported by the Research Foundation – Flanders (FWO) project number G0D9323N. S. H. was supported by the Deutsche Forschungsgemeinschaft (DFG) - Project-ID 491392403 – TRR 358.  
}

\maketitle
\end{center}

\begin{abstract}
We study divisors on moduli spaces of cubic fourfolds with simple singularities and of quasi-polarized K3 surfaces of degree $2d$. For the moduli space of cubic fourfolds, we introduce a slope quantity to characterize the effective cone and prove an explicit bound for it. For the K3 moduli spaces, we give an explicit finite presentation of the rational Picard group by showing that it is generated by Noether–Lefschetz divisors of discriminant less than or equal to $4d$. As a byproduct, we obtain two explicit expressions for the Hodge class in terms of Noether--Lefschetz divisors, and we indicate analogous results for higher-codimension Noether--Lefschetz cycles.
\end{abstract}

\setcounter{tocdepth}{1} 

\section{Introduction}

The study of the geometry of divisors on moduli spaces has been a guiding force in the subject. One among the features that distinguishes moduli spaces coarsely represented by orthogonal Shimura varieties $\mathcal{D}\big/\Gamma$, such as those of K3 surfaces, hyperk\"{a}hler varieties, or cubic fourfolds, is that they naturally come with a countable collection of modular divisors called {\it{Noether--Lefschetz divisors}}. They not only provide a rich supply of effective divisors, but (under mild hypotheses) they generate the rational Picard group of the corresponding moduli spaces, see \cite{MP13}*{Conjecture 3} and \cites{GLT15,BLMM17}. Further, the cone spanned by positive linear combinations of Noether–Lefschetz divisors is rational polyhedral \cite{BM19}, and it can be computed explicitly \cite{BBFW25}, yet little is known about the full cone of effective divisors of these moduli spaces. In this note, we refine these results in a quantitative direction. More precisely, we are primarily concerned with two questions:
\begin{enumerate}
\item When the Picard rank of the moduli space is two, such as the case of cubic fourfolds, can we determine the effective cone $\mathrm{Eff}(\mathcal{D}/\Gamma)$?
\item Can we write down a simple and uniform finite presentation for the Picard groups, allowing us to access divisor class computations by means of test curves?
\end{enumerate}

In analogy with the study of effective divisors on the moduli space of principally polarized abelian varieties $\mathcal{A}_g$, we introduce the slope $s(\mathcal{M})$ of the moduli space of cubic fourfolds with simple singularities and prove an explicit bound for it. Determining the exact value of the slope $s(\mathcal{M})$ is equivalent to determining the effective cone. For the second question, our main space of interest is the moduli space $\mathcal{F}_{2d}$ of quasi-polarized K3 surfaces of degree $2d$, where we show that the Picard group is generated by $d+1$-many NL divisors; all those of discriminant at most $4d$. Our answers to both questions rely only on the theory of modular forms associated with orthogonal Shimura varieties and can be applied to a general orthogonal modular variety.

\subsection{Slope of the moduli space of cubic fourfolds} The effective cone of divisors on natural compactifications of $\mathcal{M}_g$ or $\mathcal{A}_g$ governs significant aspects of their birational geometry, yet its description remains a challenge in both cases. For $\overline{\mathcal{A}}_g$ the cone is known when $g\leq 5$, see \cites{SM92, FGSMV14}. For $\overline{\mathcal{M}}_g$, the full cone is only known for $g\leq 3$, see \cite{Rul01}, and when restricted to the $2$-plane $\Pi$ in ${\rm{Pic}}_{\mathbb{R}}\left(\overline{\mathcal{M}}_g\right)$ spanned by the Hodge class and the full boundary, despite being studied for several decades, the cone of effective divisors in $\Pi$ is only known for $g\leq 11$, see \cites{HM90, CR91, Tan98, FP05, CHS08}.

Let $\mathcal{M}_0$ be the GIT moduli space of smooth cubic fourfolds and $\mathcal{M}$ the partial compactification by allowing at worst simple singularities. As for K3 surfaces, the Torelli theorem for smooth cubic fourfolds \cites{Voi86, Voi08} identifies a moduli point $[X]\in \mathcal{M}_0$ with the monodromy orbit of the period line $H^{3,1}(X)\in \mathcal{D}$ after choosing a marking $H^4(X,\mathbb{Z})_{\rm{prim}}(-1)\cong \Lambda_{\rm{cubic}}$. This defines an open embedding $\mathcal{M}_0\hookrightarrow\mathcal{D}\big/\Gamma$ that extends to $\mathcal{M}$, see \cite{Loo09}*{Theorem 4.1} and \cite{Laz10}*{Theorem 1.1}. The extended {\textit{period map}} 
\[
\mathcal{P}:\mathcal{M}\longrightarrow\mathcal{D}\big/\Gamma
\]
is an open embedding whose image is the complement of an irreducible NL divisor. Following \cite{Has00} we denote this divisor by $\mathcal{C}_2$. We see $\mathcal{D}/\Gamma$ as a partial compactification of $\mathcal{M}$ by adding the divisor $\mathcal{C}_2$ to the boundary. By \cite{Laz10}*{Section 6}, this boundary divisor $\mathcal{C}_2$ gets contracted after a small map resolving the rational map $\mathcal{D}\big/\Gamma\dashrightarrow\overline{\mathcal{M}}$ to the GIT compactification of $\mathcal{M}$. In particular, it is extremal. 

The Picard group ${\rm{Pic}}_{\mathbb{Q}}\left(\mathcal{M}\right)$ is one-dimensional generated by the GIT descend of $\mathcal{O}(1)$ or, alternatively, by the class of the Hodge line bundle $\lambda=c_1\left(\mathbb{E}\right)$ whose fiber over a cubic $\mathcal{P}\left(\left[X\right]\right)$ is the period line $H^{3,1}(X)$. Recall that $\lambda$ is defined on $\mathcal{D}\big/\Gamma$ and the pull-back to $\mathcal{M}$ is isomorphic to $\mathcal{O}(2)$, see \cite{Loo09}*{Theorem 4.1}. Then, the Picard group of the partial compactification ${\rm{Pic}}_{\mathbb{Q}}\left(\mathcal{D}\big/\Gamma\right)$ is two-dimensional generated by the boundary $\mathcal{C}_2$ and $\lambda$. In analogy with $\mathcal{A}_g$ we define the slope of an effective divisor $E$ in $\mathcal{M}$ as
\begin{equation}
\label{intro:eq:slopeE}
s(E)=\begin{cases}\frac{\alpha}{\beta},\;\;\; \hbox{if }\alpha,\beta>0 \\
\infty,\;\;\;\hbox{otherwise,}
\end{cases}
\end{equation}
where $\overline{E}=\alpha \lambda-\beta \mathcal{C}_2$ and $\overline{E}$ is the closure of $E$ in $\mathcal{D}\big/\Gamma$. Similarly, we define the slope of the moduli space of cubic fourfolds $\mathcal{M}$ as
\begin{equation}
\label{intro:eq:slopeM}
s(\mathcal{M}):=\inf\left\{s(E)\left|\;\; E\hbox{ effective divisor in }\mathcal{M}\right.\right\}. 
\end{equation}

A natural geometric divisor is given by the closure in $\mathcal{D}\big/\Gamma$ of the locus of singular cubics $\mathcal{M}\setminus \mathcal{M}_0$. This defines an NL divisor denoted $\mathcal{C}_6$ with slope $s\left(\mathcal{C}_6\right)=16/9$, see Equation \eqref{sec4:eq:CCL}. From \cite{BBFW25} (see Section \ref{sec:cubics}) it follows that any Noether--Lefschetz divisor $D$ satisfies 
\[
\frac{16}{9}\leq s(D).
\]
Using pull-back formulas for special cycles (see \cite{Kud97} and \cite{BFZ25}*{Proposition 2.2}) and the modularity of the generating series of volumes of Heegner divisors, see \cite{Kud03}*{Theorem I} we obtain the following lower bound on the slope of the moduli space of cubic fourfolds.

\begin{theorem}
\label{intro:thm:s(M)}
The slope of the moduli space of cubic fourfolds satisfies the following bound:
\[
\frac{523777}{206215591}\leq s(\mathcal{M})\leq 16/9.
\]
\end{theorem}

\subsection{Finite presentation of ${\rm{Pic}}_{\mathbb{Q}}\left(\mathcal{F}_{2d}\right)$}

Noether--Lefschetz divisors on moduli spaces $\mathcal{F}_{2d}$ of quasi-polarized K3 surfaces of degree $2d$ are defined by prescribing a Picard lattice on the K3 surface. Let $(L,h)$ be a fixed rank two lattice with a distinguished primitive vector $h\in L$ of square $\langle h,h\rangle=2d$. The Noether--Lefschetz divisor ${\rm{NL}}_{(L,h)}$ is defined as the locus in $\mathcal{F}_{2d}$ of K3 surfaces $(S,H)$ that admit a primitive embedding $j:L\hookrightarrow {\rm{Pic}}(S)$ with $h\mapsto H$. This defines an irreducible and reduced divisor. If we write the intersection matrix with respect to a basis as
\[
L=\left(\begin{array}{cc}2d&a\\a&2b-2\end{array}\right),\;\;h=\left(\begin{array}{c}1\\0\end{array}\right),
\]
then ${\rm{NL}}_{(L,h)}$ only depends on the discriminant $-\left|L\right|$ and the class of $a$ modulo $2d$ up to sign, see Section \ref{section2} for details. We will write ${\rm{NL}}_{(L,h)}=P_{\Delta,\delta}$, with $4d\cdot\Delta=-\left|L\right|$ and $\delta\in\mathbb{Z}\big/2d\mathbb{Z}$ the class of $a\mod 2d$ with the convention that $P_{\Delta,\delta}=P_{\Delta, -\delta}$ and $P_{\Delta,\delta}=2{\rm{NL}}_{(L,h)}$ when $\delta=-\delta$. Note that the Hodge Index Theorem implies that $-\left|L\right|>0$. Running over all such possible $(L,h)$'s defines an infinite collection of divisors whose union is analytically dense. Relaxing the primitivity condition of $j:L\hookrightarrow{\rm{Pic}}\left(S\right)$ leads to an analogous notion of NL-divisor often denoted $D_{a,b}$. This is defined as the locus of K3 surfaces $(S,H)\in \mathcal{F}_{2d}$ for which there exists a class $\beta\in {\rm{Pic}}\left(S\right)$ not proportional to $H$ with $\beta^2=2b-2$ and $\beta\cdot H=a$. Again, the class $D_{a,b}$ depends only on the discriminant $a^2-4d(b-1)$ and the class of $a$ modulo $2d$. We write $D_{a,b}=H_{m,\mu}$ with $m=\frac{a^2}{4d}-(b-1)>0$ and $\mu=a\mod 2d$. The indexing set $(m,\mu)$ is much better suited for writing cleaner formulas, as they translate to Fourier coefficients of vector-valued modular forms. 

The Picard group ${\rm{Pic}}_{\mathbb{Q}}\left(\mathcal{F}_{2d}\right)$ is spanned by both (see Equation \eqref{sec2:eq:H=P}) the set of all $D_{a,b}$'s and the set of all ${\rm{NL}}_{(L,h)}$'s. See \cite{MP13}*{Conjecture 3} and \cite{BLMM17}*{Theorem 1.1}. This gives a presentation of ${\rm{Pic}}_{\mathbb{Q}}\left(\mathcal{F}_{2d}\right)$ in terms of infinitely many generators with infinitely many linear relations given by the vanishing of Borcherds products of weight zero, see \cite{Bor 98}*{Theorem 13.3} and \cite{Bru02}*{Theorem 5.12}. A motivation for this note is to exhibit a uniform finite presentation for ${\rm{Pic}}_{\mathbb{Q}}\left(\mathcal{F}_{2d}\right)$, allowing access to divisor computations for high $d$ by means of test curves. 

\begin{theorem}
\label{intro:thm:main1}
The Picard group of the moduli space of quasi-polarized K3 surfaces ${\rm{Pic}}_{\mathbb{Q}}\left(\mathcal{F}_{2d}\right)$ is generated as $\mathbb{Q}$-vector space by the finite set of irreducible divisor classes
\begin{equation}
\label{intro:eq:thm1}
\left\{{\rm{NL}}_{(L,h)}\left|\;{\rm{rk}}(L)=2\hbox{ and }-\left|L\right|\leq 4d\right.\right\}.
\end{equation}
\end{theorem}

The set \eqref{intro:eq:thm1} generates the same $\mathbb{Q}$-vector space as $\left\{H_{m,\mu}\left|\;m\leq 1\right.\right\}$, see Equation \eqref{sec2:eq:P}, and can be rewritten (up to some factors of two when $2\delta=0$) as 
\[
\left\{\left.P_{m_\delta,\delta}\right|\;\delta=0,\ldots,d \mod 2d\right\},
\]
where $m_\delta$ is the fractional part of $\delta^2/4d$ when $\delta\neq 0$ and $m_\delta=1$ when $\delta=0$. The cardinality of \eqref{intro:eq:thm1} is $d+1$ whereas the dimension of ${\rm{Pic}}_{\mathbb{Q}}\left(\mathcal{F}_{2d}\right)$ is of the form
\begin{equation}
\label{intro:eq:Picrk}
\rho\left(\mathcal{F}_{2d}\right)=\frac{19}{24}\cdot d + O\left(\sqrt{d}\log(d)\right),
\end{equation}
see \cites{Bru02b,BLMM17}. It is worth noting that in \cite{Zha09}*{Section 3} it is shown that ${\rm{Pic}}_{\mathbb{Q}}\left(\mathcal{F}_{2d}\right)$ is generated by the Hodge class and $22d-1$ non-trivial Noether–Lefschetz divisors of small discriminant. The generating set \eqref{intro:eq:thm1} is purely geometric, no Hodge class appears, in contrast to Zhang’s generating set.

The Hodge class $\lambda=c_1\left(\mathbb{E}\right)$ is defined as the first Chern class of the Hodge line bundle $\mathbb{E}$ whose fiber over $(S,H)\in\mathcal{F}_{2d}$ is the period line $H^0\left(S,\omega_S\right)$. It was first proved in \cite{BKPSB98} that $\lambda$ is supported on NL divisors, yet no explicit formula was given. In \cite{Mau14}*{Theorem 3.1}, the result was strengthened by showing that $\lambda$ is supported on any given set of infinitely many NL divisors. The Hodge class being of Noether--Lefschetz type was reproved by different methods again in \cite{PY20}*{Section 7} and \cite{FR20}*{Theorem 1.4}, yet no explicit formulas were provided. We exhibit two types of formulas for the Hodge class in terms of the first $d+1$ NL-divisors. The first one is an explicit formula in terms of Fourier coefficients of a vector-valued Eisenstein series whose values were computed in \cite{BK01}. The second type is in terms of the representation numbers of rank $7$ definite lattices, that is, the Fourier coefficients of theta series. 

\begin{theorem}
\label{intro:thm:thm2}
The following equality holds in ${\rm{Pic}}_{\mathbb{Q}}\left(\mathcal{F}_{2d}\right)$
\begin{equation}
\label{intro:eq:lambda}
C\cdot\lambda=H_{1,0}+\sum_{\delta=1}^d a_\delta H_{m_\delta,\delta}
\end{equation}
with $C=24+c_{1,0}(E)$, and $a_{\delta}=c_{1-m_\delta,\delta}(E)$, where $E$ is the weight $7/2$ vector-valued Eisenstein series with respect to the Weil representation attached to $A_{\Lambda}=\Lambda^\vee/\Lambda$, where $\Lambda=U^{\oplus 3}\oplus\mathbb{Z}\ell^*$ with $\langle\ell^*,\ell^*\rangle=2d$ and $\delta$ is identified with $\delta\cdot\ell^*/2d+\Lambda$ in $\Lambda^\vee\big/\Lambda$. 
\end{theorem}

Explicit formulas for $c_{1-m_\delta,\delta}(E)$ can be found in \cite{BK01}*{Theorem 7}. They are of arithmetic nature and depend on the prime factors of $(1-m_\delta)\cdot 2d$ as well as invariants of $\Lambda\otimes\mathbb{Z}_p$. In particular, they are all positive rational numbers. The fact that $\lambda$ can be expressed as a positive $\mathbb{Q}$-linear combination of Heegner divisors is already implied by the proof of \cite{BM19}*{Proposition 3.3}.

Let $\Lambda$ be a negative-definite lattice with quadratic form $q$, and $\mu+\Lambda\in A_{\Lambda}$ a fixed coset. We will denote by $\theta_{m,\mu}$ the $(m,\mu)$-Fourier coefficient of the vector-valued theta series $\Theta_\Lambda$ attached to $\Lambda$, that is
\[
\theta_{m,\mu}(\Lambda)=\#\left\{x\in \Lambda^\vee\left|\;x\equiv \mu\mod \Lambda\;\;\hbox{and}\;\;-q_\Lambda(x)=m\right.\right\}.
\]

\begin{theorem}
\label{intro:thm:thm3}
Let $\ell$ be a primitive element in $E_8(-1)$ with $\langle \ell,\ell\rangle=-2d$, and $\Lambda=\ell^\perp$ the orthogonal complement. We fix an isomorphism $(A_{\mathbb{Z\ell}},-q_{\mathbb{Z}\ell})\cong(A_{\Lambda},q_\Lambda)$. Then equality \eqref{intro:eq:lambda} holds in ${\rm{Pic}}_{\mathbb{Q}}\left(\mathcal{F}_{2d}\right)$ with 
\[
C=24+\theta_{1,0}(\Lambda),\;\;\;\hbox{and}\;\;\; a_\delta=\theta_{1-m_\delta,\delta}(\Lambda),
\]
where $\delta$ is identified with $\delta\cdot e$ and the element $e\in A_\Lambda$ is a generator with $q_\Lambda(e)=1-1/2d$.
\end{theorem}

Theorem \ref{intro:thm:main1} as well as the formulas for the Hodge class generalize to the broader context of orthogonal modular varieties. These include also moduli spaces of lattice polarized K3 surfaces, (special) cubic fourfolds, and (lattice) polarized hyperk\"{a}hler varieties. See Theorem \ref{sec2:thm:general} and Corollary \ref{sec2:thm:lambda}.

\subsection{Comments on the proofs} Theorem \ref{intro:thm:s(M)} follows from the observation that on a suitable compactification, if $\gamma$ is a covering curve for an irreducible NL divisor, then
\[
\frac{\mathcal{C}_2\cdot\gamma}{\lambda\cdot \gamma}\leq s\left(\mathcal{M}\right).
\]
This is the content of Lemma \ref{sec4:lemma:gamma}. The curve and the computation of the intersection numbers are obtained by using the pull-back formulas in \cite{Kud97} as well as the generating series for volumes of special cycles. Theorems \ref{intro:thm:main1}, \ref{intro:thm:thm2}, and \ref{intro:thm:thm3} follow from elementary considerations on modular forms after interpreting algebraic cycles on $\mathcal{F}_{2d}$ as linear functionals on vector-valued modular forms. They are all consequences of the geometric theta correspondence in the context of orthogonal Shimura varieties. Concretely, the formulas for the Hodge class follow from the residue pairing defined in \cite{Bor99}*{Theorem 3.1}, and the bounds on the generating set of the Picard follow from the slope considerations in \cite{BW15}*{Section 2}. 


\subsection*{Acknowledgments} We are most grateful to Jan Bruinier, Gavril Farkas, Laure Flapan, Zhiyuan Li, Julian Lyczak, and Ricardo Zuffetti for valuable discussions related to this work. S.H. is especially grateful to Zhiyuan Li for raising the question about finding an explicit generating set of Noether–Lefschetz divisors.

\section{Preliminaries}
\label{sec1}

\subsection{Noether-Lefschetz divisors}
We fix the K3 lattice $\Lambda_{K3}:=U^{\oplus 3}\oplus E_8(-1)^{\oplus 2}$. Let $L\subset \Lambda_{K3}$ be a primitive sublattice of signature $(1,{\rm{rk}}(L)-1)$ and $h\in L$ a primitive class with $q(h)>0$. An $(L,h)$-quasipolarized K3 surface is a K3 surface $X$ together with a primitive lattice embedding $j:L\hookrightarrow {\rm{Pic}}\left(X\right)$ such that $j(h)$ is a quasi-polarization (big and nef), and for which there exists a marking 
\[
\phi:\Lambda_{K3}\overset{\sim}{\longrightarrow}H^2\left(X,\mathbb{Z}\right)
\]
such that $\phi\mid_L=j$, see \cites{Dol96, AE25}, see also \cite{BBV25}*{Definition 2.1}. When $L=\mathbb{Z}h$ has rank one, we call it an $h$-polarization. There is a moduli space $\mathcal{F}_{(L,h)}$ of $(L,h)$-quasipolarized K3 surfaces and a natural (forgetful) proper morphism
\begin{equation}
\label{eq:sec2:FL}
\mathcal{F}_{(L,h)}\longrightarrow\mathcal{F}_{h}.
\end{equation}
The reduced cohomology class of the image is denoted by
\begin{equation}
\label{eq:sec2:NL_Lh}
{\rm{NL}}_{(L,h)}\in H^{2{\rm{rk}}(L)-2}\left(\mathcal{F}_h,{\mathbb{Q}}\right).
\end{equation}
At the level of periods, one has that if $\mathcal{D}_{(L,h)}=\mathcal{D}_{(L,h)}^+$ is a component of the analytic open
\[\mathcal{D}_{(L,h)}^+\cup\mathcal{D}_{(L,h)}^-=
\left\{x\in \mathbb{P}\left(L^{\perp \Lambda_{K3}}\otimes \mathbb{C}\right)\left|\; q(x)=0\hbox{ and }\langle x, \overline{x}\rangle>0\right.\right\},
\]
then \eqref{eq:sec2:FL} on periods becomes
\[
\Phi_{(L,h)}:\mathcal{D}_{(L,h)}\big/\Gamma_{(L,h)}\longrightarrow\mathcal{D}_h\big/\Gamma_h,
\]
where $\Gamma_{(L,h)}$ is the arithmetic group consisting of isometries of $L^{\perp}$ fixing $\mathcal{D}_{(L,h)}$ and acting trivially on the discriminant $A_{L^\perp}=\left(L^{\perp}\right)^\vee\big/L^{\perp}$. The modular variety $\mathcal{D}_{(L,h)}\big/\Gamma_{(L,h)}$ coarsely represents $\mathcal{F}_{(L,h)}$ and ${\rm{NL}}_{(L,h)}$ coincides with the reduced class of the image of $\Phi_{(L,h)}$. Further, the spaces $\mathcal{F}_{(L,h)}$ and $\mathcal{F}_{(L',h')}$ are the same if there exists $g\in{\rm{O}}\left(\Lambda_{K3}\right)$ such that $g(L)=L'$ and $g(h)=h'$. In particular, $\mathcal{F}_{h}$ depends only on the degree $2q(h)=\langle h,h\rangle=2d$, and if we fix $h\in\Lambda_{K3}$, then
\begin{equation}
\label{sec2:eq:K3lattice}
\Lambda_{2d}:=h^\perp\cong U^{\oplus 2}\oplus E_8(-1)^{\oplus 2}\oplus\mathbb{Z}\ell\;\;\hbox{where}\;\;q(\ell)=-d, 
\end{equation}
and $\mathcal{F}_{(L,h)}$ depends only on the $\Gamma_h$-orbit of $h^{\perp L}\subset\Lambda_{2d}$. For an arbitrary even lattice $\Lambda$ of signature $(2,n)$ we denote by $\mathcal{D}_{\Lambda}$ the corresponding connected type IV symmetric domain, by $\Gamma_{\Lambda}$ the arithmetic group of isometries of $\Lambda$ fixing $\mathcal{D}_{\Lambda}$ and acting trivially on $A_{\Lambda}$, and by $\mathcal{F}_{\Lambda}$ the corresponding quotient.

We briefly review the relation in codimension one between NL and {\textit{Heegner divisors}}. When ${\rm{rk}}(L)=2$ the $\Gamma_h$-orbit of $\mathbb{Z}v=h^\perp\cap L\subset \Lambda_{2d}$ depends only on the square $\left<v,v\right>$ and the class in $A_{\Lambda_{2d}}$ (up to sign) of the primitive multiple $v_*=\frac{v}{{\rm{div}}(v)}$ of $v$ in $\Lambda_{2d}^\vee\subset \Lambda_{2d}\otimes \mathbb{Q}$. We write
\[
\Delta=-q(v_*)\;\;\hbox{and}\;\;\delta=v_*+\Lambda_{2d}\in A_{\Lambda_{2d}}.
\]
Here ${\rm{div}}(v)$ is the positive generator of the ideal $\langle v,\Lambda_{2d}\rangle\subset \mathbb{Z}$. Note that in this case $A_{\Lambda_{2d}}$ is cyclic of order $2d$ generated by $\ell_*=\ell/2d$. We denote ${\rm{NL}}_{(L,h)}$ by $P_{\Delta,\delta}$ with the convention that $P_{\Delta,\delta}=P_{\Delta,-\delta}$ and $P_{\Delta,\delta}=2{\rm{NL}}_{(L,h)}$ when $\delta=-\delta$ in $A_{\Lambda_{2d}}$. The relation between the determinant $\left|L\right|$ and $\Delta\in \mathbb{Q}_{\geq0}$ is given by $-\left|L\right|=4d\Delta$. Further, if $\delta=a\ell_*\mod \Lambda_{2d}$, then $(L,h)$ can be realized as 
\begin{equation}
\label{sec2:eq:L}
L=\left(\begin{array}{cc}2d&a\\a&2b-2\end{array}\right),\;\;h=\left(\begin{array}{c}1\\0\end{array}\right).
\end{equation}
The divisor $P_{\Delta,\delta}$ can also be described as the $\Gamma_{h}$-quotient of the hyperplane arrangement 
\begin{equation}
\label{sec2:eq:P}
P_{\Delta,\delta}:=\left(\sum_{\substack{x\in\delta+\Lambda_{2d}\hbox{\tiny{ primitive}}\\ q(x)=-\Delta}}x^\perp\cap\mathcal{D}_h\right)\big/\Gamma_h
\end{equation}
called a {\textit{primitive Heegner divisor}}. The sum (instead of a union) stands here for the fact that when $\delta=-\delta$ we count hyperplanes with multiplicity two. The corresponding $\Gamma_h$-quotient of \eqref{sec2:eq:P} without the primitivity assumption is called a {\textit{Heegner divisor}} denoted $H_{m,\mu}$. The definition of both primitive and non-primitive Heegner divisors generalizes to arbitrary $(2,n)$-type lattice $\Lambda$. Since every element is a multiple of a primitive element, by M\"{o}bius inversion one obtains 
\begin{equation}
\label{sec2:eq:H=P}
H_{m,\mu}=\sum_{\substack{s>0\\ \delta\in A_\Lambda\\s\delta=\mu}}P_{\frac{m}{s^2},\delta}\;\;\;\hbox{and}\;\;\;P_{\Delta,\delta}=\sum_{\substack{s>0\\ \alpha\in A_\Lambda\\s\alpha=\delta}}\mu(s)H_{\frac{\Delta}{s^2},\alpha}
\end{equation}
where $P_{\frac{m}{s^2},\delta}=0$ (similarly $H_{\frac{\Delta}{s^2}, \alpha}$=0) if there is no primitive class $x$ in $\delta+\Lambda$ with $q(x)=m/s^2$, see \cite{BM19}*{Lemma 4.2}. When $\Lambda=\Lambda_{2d}$ and $\mu=a\ell_*$, the divisor $H_{m,\mu}$ (also denoted $D_{a,b}$) is supported on the locus of K3 surfaces $(S,H)$ admitting a class $\beta\in{\rm{Pic}}(S)$ (not necessarily primitive!) not proportional to $H$ where the intersection matrix of $\{H,\beta\}$ is given by \eqref{sec2:eq:L} and $m=a^2/4d-(b-1)$, see \cite{MP13}*{Lemma 3}. 

One declares $-H_{0,0}$ to be the Hodge class $\lambda\in{\rm{Pic}}_{\mathbb{Q}}\left(\mathcal{F}_{\Lambda}\right)$. We work with $H_{m,\mu}$ because of the formal properties of the generating series
\begin{equation}
\label{sec2:eq:KM}
\Theta_{{\rm{KM}},\Lambda}(\tau)=\sum_{m,\mu}H_{m,\mu}q^{m}\mathfrak{e}_{\mu}
\in {\rm{Pic}}_{\mathbb{Q}}\left(\mathcal{F}_{\Lambda}\right)[\![q^{1/N}]\!]\otimes \mathbb{C}\left[A_{\Lambda}\right].
\end{equation}
Here $N$ stands for the {\textit{level}} of $\Lambda$, that is, the smallest $N$ such that $N\cdot q_{\Lambda}(\cdot)$ is integral on $\Lambda^\vee$.

\subsection{Vector-valued modular forms} Let $\Lambda$ be an even lattice of signature $(b^+,b^-)$ with $\mathbb{Z}$-valued quadratic form $q_\Lambda(x)=\frac{\langle x,x\rangle}{2}$, and discriminant group $A_\Lambda$. There is a {\textit{Weil representation}} 
\[
\rho_\Lambda:{\rm{Mp}}_2(\mathbb{Z})\longrightarrow{\rm{GL}}\left(\mathbb{C}\left[A_\Lambda\right]\right)
\]
of the metaplectic double cover of ${\rm{SL}}_2(\mathbb{Z})$ on the group algebra of $A_\Lambda$. We write $\{\mathfrak{e}_\mu\left|\;\mu\in A_\Lambda\right.\}$ for the standard basis of $\mathbb{C}\left[A_\Lambda\right]$. See \cite{Bor98}*{Section 4} for a description in terms of the standard generators of ${\rm{Mp}}_2(\mathbb{Z})$. A {\textit{vector-valued modular form of weight}} $k\in\frac{1}{2}\mathbb{Z}$ with respect to the dual Weil representation $\rho_\Lambda^*$ is a holomorphic function $f:\mathbb{H}\longrightarrow\mathbb{C}\left[A_\Lambda\right]$ such that 
\begin{equation}
\label{sec2:eq:f}
f(\gamma\tau)=\phi(\tau)^{2k}\rho_\Lambda^*(\gamma)f(\tau)\;\;\hbox{for all}\;\gamma=(M,\phi)\in{\rm{Mp}}_2(\mathbb{Z})
\end{equation}
and $f$ is holomorphic at $\infty$. If $f$ vanishes at $\infty$ it is called a {\textit{cusp form}}. They form finite-dimensional vector spaces denoted ${\rm{M}}_{k,\Lambda}$ and ${\rm{S}}_{k,\Lambda}$ respectively. Every such function admits a Fourier expansion
\begin{equation}
\label{sec2:eq:fFourier}
f=\sum_{\mu\in A_\Lambda}\sum_{\substack{m\in \mathbb{Z}-q_\Lambda(\mu)\\m\geq 0}}\alpha_{m,\mu}q^m\mathfrak{e}_\mu\;\;\;\hbox{where}\;\; q=e^{2\pi i\tau}.
\end{equation}
When $\alpha_{0,\mu}=0$ for all non-trivial isotropic elements $\mu\in A_\Lambda$, the modular form $f$ is called an {\textit{almost cusp form}}. The vector space of almost cusp forms is denoted ${\rm{M}}_{k,\Lambda}^\circ\subset{\rm{M}}_{k,\Lambda}$. Assuming further that $2k\equiv b^+-b^-\mod 4$ and $k>2$, if we regard $\mathfrak{e}_0:\mathbb{H}\longrightarrow \mathbb{C}\left[A_\Lambda\right]$ as a constant function, then the Eisenstein series 
\begin{equation}
\label{sec2:eq:E}
E_{k,\Lambda}(\tau)=\frac{1}{4}\sum_{\gamma\in \Gamma_\infty\backslash {\rm{Mp}}_2\left(\mathbb{Z}\right)}\phi(\tau)^{-2k}\rho_\Lambda^*(\gamma)^{-1}\mathfrak{e}_0
\end{equation}
is a vector-valued almost cusp form of weight $k$, where $\Gamma_{\infty}$ is the stabilizer of the cusp. In particular ${\rm{M}}_{k,\Lambda}^\circ=\mathbb{C}E_{k,\Lambda}\oplus{\rm{S}}_{k,\Lambda}$. Further, when $b^+=2$ and $2k={\rm{rk}}(\Lambda)$, the Fourier coefficients of $E_{k,\Lambda}$ are rational, see \cite{BK01}. Assume $\Lambda$ has signature $(2,n)$ and $k=1+n/2$. The modularity of the corresponding generating series \eqref{sec2:eq:KM} states that $\Theta_{{\rm{KM}},\Lambda}\in {\rm{Pic}}_{\mathbb{Q}}\left(\mathcal{F}_\Lambda\right) \otimes{\rm{M}}_{k,\Lambda}^\circ\left(\mathbb{Q}\right)$, see \cite{KM90, Bor99}, where ${\rm{M}}_{k,\Lambda}^\circ\left(\mathbb{Q}\right)$ is the $\mathbb{Q}$-vector space consisting of vector-valued almost cusp forms with rational Fourier coefficients. In particular, $\Theta_{{\rm{KM}},\Lambda}$ induces a map
\begin{equation}
\label{sec2:eq:iso}
\Psi_\Lambda:{\rm{M}}_{k,\Lambda}^\circ\left(\mathbb{Q}\right)^\vee\longrightarrow{\rm{Pic}}_{\mathbb{Q}}\left(\mathcal{F}_\Lambda\right)
\end{equation}
that sends the $(m,\mu)$-Fourier coefficient extraction functional $c_{m,\mu}:f\mapsto \alpha_{m,\mu}$ to the Heegner divisor $H_{m,\mu}$. When $\Lambda$ splits off two copies of the hyperbolic plane, the map $\Psi_\Lambda$ is an isomorphism, see \cite{Bru02}*{Theorem 0.4}, \cite{Bru14}*{Theorem 1.2}, and \cite{BLMM17}*{Theorem 1.8}. See also \cite{BZ24}*{Remark 3.13}.

\subsection{Higher codimension NL cycles}

The (reducible) NL-divisor $D_{a,b}$ on $\mathcal{F}_{2d}$ generalizes naturally to higher codimension. For $\underline{a}=(a_1,\ldots,a_g)\in \mathbb{Z}^g$, and $T_0\in {\rm{Sym}}^{g}\left(\mathbb{Z}\right)$ a symmetric $g\times g$-matrix with integral entries, we denote by $D_{T_0, \underline{a}}$ the reduced cycle given by the locus of K3 surfaces $(S,H)\in\mathcal{F}_{2d}$ for which there are classes $\beta_1,\ldots,\beta_g\in{\rm{Pic}}\left(S\right)$ with no linear combination proportional to $H$ such that $H\cdot \beta_i=a_i$ and $(\left<\beta_i,\beta_j\right>) = T_0$. If non-empty, this defines an algebraic cycle of codimension given by the rank of $T_0$. In terms of irreducible NL-cycles one has that
\[
D_{T_0, \underline{a}}=\sum_{(L,h)}{\rm{NL}}_{(L,h)}\;\;\;\hbox{in}\;\;\;{\rm{CH}}^{{\rm{rk}}(T_0)}\left(\mathcal{F}_{2d}\right),
\]
where the sum runs over isomorphism classes of embeddings 
\[
(L,h)\hookrightarrow \left({\rm{Pic}}(S),H\right),\;\;\hbox{with}\;\;L=\left(\begin{array}{cc}2d&\underline{a} \\ \underline{a}^t&T_0\end{array}\right),\;\;h=\left(\begin{array}{c}1\\0\\\vdots\\0\end{array}\right),
\]
up to the action of ${\rm{O}}^{+}\left(\Lambda_{K3}, h\right)$. Note that the Hodge index theorem implies that after projecting all $\beta_i$ to $H^{\perp}$ in ${\rm{Pic}}(S)\otimes \mathbb{Q}$, the resulting intersection matrix of the projected classes $\widetilde{\beta}_i$ must be negative semi-definite.

Heegner divisors generalize further to {\it{special cycles}} which we now briefly explain. Let $\Lambda$ be an even lattice of signature $(2,n)$. We fix a positive semi-definite matrix $T\in {\rm{Sym}}^g\left(\mathbb{Q}\right)$, and class $\underline{\mu}=(\mu_1,\ldots,\mu_g)\in A_{\Lambda}^g$. One defines the special cycle $Z_{T,\underline{\mu}}\in{\rm{CH}}^{{\rm{rk}}(T)}\left(\mathcal{F}_{\Lambda}\right)$ as 
\[
Z_{T,\underline{\mu}}=\left(\sum_{\substack{\underline{x}\in\underline{\mu}+\Lambda^g\\q_\Lambda(\underline{x})=-T}}{\rm{span}}(x_1,\ldots,x_g)^\perp\cap\mathcal{D}_\Lambda\right)\big/\Gamma_\Lambda.
\]
When $\Lambda$ is the lattice \eqref{sec2:eq:K3lattice} corresponding to $\mathcal{F}_{2d}$, the same argument as in \cite{MP13}*{Lemma 3} leads to the following relation between Noether--Lefchetz and special cycles.

\begin{lemma}
    We have ${\rm{Supp}}\left(Z_{T,\underline{\mu}}\right)=D_{T_0, \underline{a}}$, where
    \[
    T = -\frac{1}{2}\left(T_0-\frac{1}{2d}\underline{a}^t\underline{a}\right) \hbox{ and } \underline{\mu}\equiv (a_1\ell_*,\dots,a_g\ell_*)\mod \Lambda_{2d}^g.
    \]
\end{lemma}

Generalizing \eqref{sec2:eq:KM}, one constructs the Kudla--Millson generating series 
\[
\Theta_{{\rm{KM}},\Lambda}^g=\sum_{\underline{\mu}\in A_{\Lambda}^g}\sum_{\substack{T\in{\rm{S}}^g-q_\Lambda(\underline{\mu})\\T\geq0}}\left(Z_{T,\underline{\mu}}\cdot(-\lambda)^{g-{\rm{rk}}(T)}\right)q^T\mathfrak{e}_{\underline{\mu}}\in {\rm{CH}}^g\left(\mathcal{F}_{\Lambda}\right)\llbracket q^T\left|T\in{\rm{Sym}}^g(\mathbb{Q})\right.\rrbracket\otimes \mathbb{C}\left[A_{\Lambda}^g\right],
\]
where $q^T = e^{2\pi i\mathrm{Tr}(T\tau)}$, and $\tau\in \mathbb{H}_g$, the Siegel upper half-space of genus $g$. The main results in \cite{KM90}, \cite{BW15}*{Theorem 6.2} and \cite{BR15Correction} is that $\Theta_{{\rm{KM}},\Lambda}^g$ is the Fourier expansion of a vector-valued Siegel modular form of weight $k={\rm{rk}}(\Lambda)/2$ with respect to the dual of the Weil representation $\rho_{\Lambda,g}:{\rm{Mp}}_{2g}(\mathbb{Z})\longrightarrow {\rm{GL}}\left(\mathbb{C}\left[A_\Lambda^g\right]\right)$, where ${\rm{Mp}}_{2g}(\mathbb{Z})$ is the metaplectic double cover of ${\rm{Sp}}_{2g}\left(\mathbb{Z}\right)$. In particular, there is a $\mathbb{C}$-linear map 
\begin{equation}
\label{sec2:eq:psi_g}
\Psi_{\Lambda}^g:\left({\rm{M}}_{k,\Lambda}^g\right)^{\vee}\longrightarrow {\rm{CH}}^{g}\left(\mathcal{F}_\Lambda\right)\otimes\mathbb{C}
\end{equation}
generalizing \eqref{sec2:eq:iso}, sending the Fourier coefficient extraction functional $c_{T,\underline{\mu}}:{\rm{M}}_{k,\Lambda}^g\longrightarrow \mathbb{C}$ to the cycle class $Z_{T,\underline{\mu}}\cdot (-\lambda)^{g-{\rm{rk}}(T)}\in {\rm{CH}}^{g}\left(\mathcal{F}_\Lambda\right)\otimes\mathbb{C}$.

\section{The Hodge class and finite generation of special cycles}
\label{section2}

Consider the Heisenberg group
$$\mathcal{H}_{g - 1,1}(\mathbb{R}) = \left\{(r, s, t) \in (\mathbb{R}^{g-1})^2 \times \mathbb{R} \bigg \vert t - r^t s \right\}$$
with multiplication
$$(r, s, t)(r', s', t') = (r+r', s+s', t + t' + r^t s' - s^t r').$$ 
Define the Jacobi group $\mathcal{J}_{g - 1, 1}(\mathbb{R}) = \rm{Mp}_{2 g - 2}(\mathbb{R}) \ltimes \mathcal{H}_{g - 1,1}(\mathbb{R})$, where $M \in \rm{Mp}_{2g - 2}(\mathbb{R})$ acts on $(r, s, t) \in \mathcal{H}_{g - 1,1}(\mathbb{R})$ via
$$M (r, s, t) = (ar + bs, cr + ds, t).$$
We have an embedding of $\mathcal{J}_{g - 1, 1}(\mathbb{R})$ into $\rm{Mp}_{2g}(\mathbb{R})$ via
$$((M, \phi), (r, s, t)) \mapsto \left(\begin{pmatrix} a & 0 & b & as - bs \\ r^t & 1 & s^t & t \\ c & 0 & d & cs - dr \\ 0 & 0 & 0 & 1 \end{pmatrix}, \tilde{\phi}\right),$$
where $\tilde{\phi}(\tau) = \phi(\tau_1)$. The Weil representation $\rho_{\Lambda, g}$ induces representations $\rho_{\Lambda, g - 1, \mu_2}$ for $\mu_2 \in A_\Lambda$ satisfying
$$\rho_{\Lambda, g}(M) \mathfrak{e}_{(\mu_1, \mu_2)} = (\rho_{\Lambda, g - 1, \mu_2}(M) \mathfrak{e}_{\mu_1}) \otimes \mathfrak{e}_{\mu_2}$$
for all $M \in \mathcal{J}_{g - 1, 1}(\mathbb{Z})$ and $\mu_1 \in A_\Lambda^{g-1}, \mu_2 \in A_\Lambda$. A vector-valued Jacobi form of weight $k$ and index $m$ with respect to $\rho^*_{\Lambda, g-1, \mu_2}$ is a holomorphic function $f : \mathbb{H}_{g-1} \times \mathbb{C}^{g-1} \to \mathbb{C}[A_\Lambda^{g-1}]$ that satisfies
$$f(\tau_1, z + r \tau_1 + s) = e(-r m r^t \tau_1 - 2 r^t m z - r^t m s) \rho^*_{\Lambda, g-1, \mu_2}(r, s, t) f(\tau_1, z)$$
for all $(r, s, t) \in \mathcal{H}_{g-1,1}(\mathbb{Z})$ and
$$f(M(\tau_1, z)) = \phi(\tau_1)^{2k} e(z m \tau_2^t c (c \tau_1 + d)^{-1}) \rho^*_{\Lambda, g-1, \mu_2}(M) f(\tau_1, z)$$
for all $M \in \rm{Mp}_{2 g-2}(\mathbb{Z})$. Then $f$ has a Fourier expansion
$$f(\tau_1, z) = \sum_{\substack{\mu_1 \in A_\Lambda^{g-1} \\ T_1 \in {\rm{S}}^{g-1} - q_\Lambda(\mu_2) \\ T_{12} \in \mathbb{Z}^{g-1} - (\mu_1, \mu_2)_\Lambda}} c_{T_1, T_{12},\mu_1}(f) \mathfrak{e}_{\mu_1}(T_1 \tau_1 + T_{12} z).$$

Let $f$ be a vector-valued Jacobi form of weight $k$ and index $m$ with respect to $\rho^*_{\Lambda, g-1, \mu_2}$ and let $\ell = \operatorname{ord}(\mu_2)$. Then $(U_\ell f)(\tau_1, z) = f(\tau_1, \ell z)$ is a vector-valued Jacobi form of weight $k$ and index $\ell^2 m$ with respect to $\rho^*_{\Lambda, g-1}$ and the Fourier coefficients are related by $c_{T_1, T_{12}, \mu_1}(f) = c_{T_1, \ell T_{12}, \mu_1}(U_\ell f)$. Let $A_{\ell^2 m}$ be the discriminant group associated to the lattice $\mathbb{Z}$ with quadratic form $q_{\ell^2 m}(x) = - \ell^2 m x^2$ and define for $a \in A_{\ell^2 m}^{g-1}$ the Jacobi theta function
$$\Theta_{\ell^2 m, a}(\tau_1, z) = \sum_{\lambda \in a} e(-q_{\ell^2 m}(\lambda) \tau - (\lambda, z)_{\ell^2 m}).$$
Then $U_\ell f$ has a theta decomposition
$$(U_\ell f)(\tau_1, z) = \sum_{a \in A_{\ell^2 m}^{g-1}} \psi_a(\tau_1) \Theta_{\ell^2 m, a}(\tau_1, z)$$
and $\psi = (\psi_a)_{a \in A_{\ell^2 m}^{g-1}}$ is a vector-valued Siegel modular form of weight $k - 1/2$ with respect to the dual Weil representation of $\Lambda \oplus \langle - \ell^2 m \rangle$, see \cite{Wil19} for $g = 1$ and \cite{Ziegler} for the general case. The Fourier coefficients of $\psi$ are then given by
$$c_{T_1 - T_{12} T_{12}^t / (4m), (\mu_1, T_{12})}(\psi) = c_{T_1, T_{12}, \mu_1}(f).$$

Let now $f \in {\rm{M}}_{k,\Lambda}^g$. Then $f$ has a Fourier-Jacobi expansion
\[
f(\tau)=f\left(\begin{array}{cc}\tau_1&z\\z^t&\tau_2\end{array}\right) = \sum_{\substack{\mu_2 \in A_\Lambda \\ m \in \mathbb{Z} - q_\Lambda(\mu) \\ m \geq 0}} f_{m,\mu_2}(\tau_1, z) \otimes \mathfrak{e}_{\mu_2}(m \tau_2),
\]
where $f_{m, \mu_2}$ is a vector-valued Jacobi form of weight $k$, index $m$ with respect to the Weil representation $\rho^*_{\Lambda, g-1, \mu_2}$, see \cite[Section 3.5]{KieferZuffetti} for $g = 2$. The Fourier coefficients of $f$ and $f_{m, \mu_2}$ are then related via
$$c_{T, \underline{\mu}}(f) = c_{T_1, T_{12}, \mu_1}(f_{m, \mu_2}),$$
where $\underline{\mu} = (\mu_1, \mu_2), \mu_1 \in A_\Lambda^{g-1}, \mu_2 \in A_\Lambda$ and $T = \begin{pmatrix}T_1 & T_{12} / 2 \\ T_{12}^t / 2 & m\end{pmatrix}$. We consider the Siegel modular forms $\psi_{m, \mu_2}$ associated to $f_{m, \mu_2}$. Then we have for the Fourier coefficients
\begin{equation}
    c_{T, \underline{\mu}}(f) = c_{T_1 - T_{12} T_{12}^t / (4m), (\mu_1, T_{12})}(\psi_{m, \mu_2}). \label{eq:ThetaDecompFC}
\end{equation}

We write $\operatorname{ord}(f)$ for the smallest $m$ with $f_{m,\mu_2} \neq 0$. Recall that the slope of a weight $k$ scalar-valued Siegel modular form $f$ for the full ${\rm{Sp}}_{2g}(\mathbb{Z})$ is defined as the ratio $s(f)=k/{\rm{ord}}(f)$, and the slope $s_g$ of  the moduli space of principally polarized abelian varieties $\mathcal{A}_g$ is defined as \[
s_g:=\inf\left\{s(f)\left|\;f\in{\rm{M}}_{\bullet}^g\setminus\{0\}\right.\right\}.
\] 
The smallest possible slope for an elliptic modular form is $12$ and attained by the discriminant form $\Delta(\tau)$. By \cite{BW15}*{Proposition 1.4}, we have $\operatorname{ord}(f) \leq \frac{k}{s_g}$.

Let $C_{1,g}(k) = \frac{k}{s_g}$ and define recursively
$$C_{i+1, g}(k) = \max\left(C_{1, g}(k), C_{i,g-1}\left(k-\frac{1}{2}\right)+ \frac{1}{4} C_{1,g}(k) \right).$$
A straight-forward induction yields the more explicit bound
$$C_{i, g}(k) \leq \sum_{j = 0}^{i-1} \frac{k-\frac{j}{2}}{s_{g-j}}.$$
\begin{proposition}
For a symmetric positive semi-definite matrix $T \in {\rm{Sym}}^{g}\left(\mathbb{Q}\right)$ we write $\lambda_i(T)$ for its $i$-th successive minimum, that is, the smallest number such that there are $i$-many linear independent integral vectors $v_1, \ldots, v_i \in \mathbb{Z}^g$ with $T(v_j, v_j) \leq \lambda_i(T)$. Let $f \in {\rm{M}}_{k,\Lambda}^g$ and assume that the Fourier coefficients $c_{T, \underline{\mu}}(f)$ vanish for all $T \geq 0$ with $\lambda_i(T) \leq C_{i, g}(k)$. Then $f = 0$.
\end{proposition}
\begin{proof}
We proof this by induction on the genus $g$. For $g = 1$, this directly follows from \cite{BW15}*{Proposition 2.4}. Assume now $g > 1$ and let $f \neq 0$. We show that there is some index $T, \underline{\mu}$ with $\lambda_i(T) \leq C_{i, g}(k)$ and $c_{T, \underline{\mu}}(f) \neq 0$. According to \cite{BW15}*{Proposition 2.4}, there is some $m \leq \frac{k}{s_g} = C_{1,g}(k)$ with $f_{m, \mu_2} \neq 0$. Using the theta decomposition of Siegel-Jacobi forms, the function $f_{m, \mu_2}$ can be identified with a vector-valued Siegel modular form $\psi_{m, \mu_2} \in {\rm{M}}_{k - \frac{1}{2},\Lambda \oplus \langle \ell^2 m\rangle}^{g-1}$, where $\ell$ is the order of $\mu_2$ and the Fourier coefficients are related by \eqref{eq:ThetaDecompFC}. By the non-triviality of $f_{m, \mu_2}$ and the induction hypothesis, there is some index $T_1, (\mu_1, T_{12})$ with $\lambda_i(S) \leq C_{i, g-1}(k - 1/2)$ such that $c_{S, (\mu_1, T_{12})}(\psi_{m, \mu_2}) = c_{T, \underline{\mu}}(f) \neq 0$, where $S = T_1 - T_{12} T_{12}^t / (4m)$. Now,
$$\lambda_{i+1}(T) \leq \max\left(m, \lambda_i(S) + \frac{m}{4} \right) \leq C_{i+1, g}(k)$$
for $1 \leq i \leq g$ and also $\lambda_1(T) \leq m \leq C_{1,g}(k)$, which proves the proposition.
\end{proof}
\begin{remark}
Since $c_{A^{t}TA, \underline{\mu} A}(f) = c_{T, \underline{\mu}}(f)$ for any $A \in \operatorname{GL}_g(\mathbb{Z})$, it suffices to check the vanishing of $c_{T, \underline{\mu}}(f)$ for a set of representatives of the finite set
\[
S_{k,g, \Lambda} = \left\{ (T, \underline{\mu}) \in {\rm{Sym}}^g\left(\mathbb{Q}\right)\times A_\Lambda^g\ \bigg\vert\ T \in {\rm{S}}^g - q_\Lambda(\underline{\mu}), T \geq 0, \lambda_i(T) \leq C_{i, g}(k) \right\} \big/ \operatorname{GL}_n(\mathbb{Z}),
\]
\end{remark}
An immediate consequence of \cite{BW15}*{Proposition 2.4} and the existence of the linear map \eqref{sec2:eq:psi_g} is the following finite generation result, generalizing \cite{Zha09}*{Section 3}.
\begin{theorem}[Bruinier--Westerholt-Raum]
The subspace ${\rm{SC}}^g(\mathcal{F}_\Lambda)\subset{\rm{CH}}_{\mathbb{Q}}^g\left(\mathcal{F}_{\Lambda}\right)$ generated by the cycles $Z_{T,\underline{\mu}}\cdot \lambda^{g-{\rm{rk}}(T)}$ is finite dimensional generated by the finite set
\[
\left\{Z_{T,\underline{\mu}}\cdot \lambda^{g-{\rm{rk}}(T)}\ \bigg\vert\ T \in S_{\frac{{\rm{rk}}(\Lambda)}{2}, g, \Lambda}\right\}.
\]
\end{theorem}
\begin{proof}
We fix $k={\rm{rk}}(\Lambda)/2$. Via the $\mathbb{C}$-linear map \eqref{sec2:eq:psi_g} it is enough to show that $\left({\rm{M}}_{k,\Lambda}^g\right)^\vee$ is generated by the coefficient extraction functionals $c_{T,\underline{\mu}}$ with $T \in S_{k, g, \Lambda}$, which follows from the previous proposition. We conclude by observing that the image of the linear map \eqref{sec2:eq:psi_g} is ${\rm{SC}}^g\left(\mathcal{F}_{\Lambda}\right)$.
\end{proof}

As a corollary we have:

\begin{corollary}
\label{sec2:prop:generation}
Let $\Lambda$ be an even lattice of signature $(2,n)$ with $n\geq3$ splitting off two copies of the hyperbolic plane. Then ${\rm{Pic}}_{\mathbb{Q}}\left(\mathcal{F}_\Lambda\right)$ is generated by the set
\begin{equation}
\label{sec2:eq:genH}
\left\{H_{m,\mu}\left|\;\mu\in A_\Lambda, 0\leq m\leq \frac{{\rm{rk}}(\Lambda)}{24}\right.\right\}.
\end{equation}
\end{corollary}

\begin{proof}
The corollary follows from the surjectivity of \eqref{sec2:eq:iso} and the fact that $s_1= 12$.
\end{proof}

\begin{remark}
\label{sec2:rmk:H=P}
We would like to stress again that $H_{0,0}$ is not a divisor, but a class defined as $-\lambda$ by convention. Note also that by \eqref{sec2:eq:H=P}, the set \eqref{sec2:eq:genH} with Heegner divisors $H_{m,\mu}$ replaced by primitive Heegner divisors $P_{m,\mu}$ is also a generating set for ${\rm{Pic}}_{\mathbb{Q}}\left(\mathcal{F}_\Lambda\right)$. 
\end{remark}

Under the light of \eqref{sec2:eq:iso}, relations among Fourier coefficients of modular forms in ${\rm{M}}_{k,\Lambda}^\circ$ lead to relations among Heegner divisors. One of the most fundamental sources of relations among Fourier coefficients of modular forms comes from the Residue Theorem. See for instance \cite{Bor99}*{Theorem 3.1}. This can be seen as an instance of Serre duality and they emanate from the fact that if $f\in {\rm{M}}_k^!(\rho)$ is a vector-valued weakly holomorphic modular form of weight $k$ with respect to some finite-dimensional representation $\rho$ of ${\rm{Mp}}_2(\mathbb{Z})$, then taking the residue at the cusp at infinity induces a pairing 
\begin{equation}
\label{sec2:eq:Res_pair}
{\rm{M}}_{k_1}^!(\rho)\times{\rm{M}}_{k_2}^!(\rho^*)\longrightarrow{\rm{Res}}_{\infty}(f(\tau)\cdot g(\tau) d\tau):=\frac{1}{2\pi i}{\rm{Res}}_0\left((f\cdot g)(q)q^{-1} dq\right).
\end{equation}
Here $\rho^*$ is the dual representation making $f\cdot g$ a scalar-valued weakly holomorphic modular form. When the weight $k_1+k_2=2$, the local form $(f\cdot g)d\tau$ defines a global meromorphic form on $\mathbb{P}^1$ with vanishing residue if the poles of $f\cdot g$ are concentrated at the cusp $\infty$. By identifying $\rho_\Lambda^*$ with $\rho_{\Lambda(-1)}$ and $(A_{\Lambda(-1)},q_{\Lambda(-1)})$ with $(A_\Lambda,-q_\Lambda)$ we obtain the following.

\begin{proposition}
\label{sec2:prop:rels}
Let $\Lambda$ be an even lattice of signature $(2,n)$, fix $k={\rm{rk}}(\Lambda)/2$ and let $f\in {\rm{M}}_{2-k,\Lambda(-1)}^!$ be a meromorphic modular form of weight $2-k$ with respect to the Weil representation $\rho_\Lambda$, and with poles concentrated at the cusp at infinity. Assume $f$ has Fourier expansion
\[
f(\tau)=\sum_{\substack{m,\mu\\
-M\leq m\\\mu\in A_{\Lambda}}}\alpha_{m,\mu}q^m\frak{e}_\mu.
\]
Then the following relation holds in ${\rm{Pic}}\left(\mathcal{F}_\Lambda\right)\otimes\mathbb{C}$:
\begin{equation}
\label{sec2:eq:univ_relation}
\sum_{\substack{0\leq m\leq M\\\mu\in A_\Lambda}}\alpha_{-m,\mu}H_{m,\mu}=0.
\end{equation}
\end{proposition}

\begin{proof}
By the Residue Theorem the pairing \eqref{sec2:eq:Res_pair} restricted to $\{f\}\times {\rm{M}}_{k,\Lambda}^\circ$ is trivial. In particular for all $g\in {\rm{M}}_{k,\Lambda}^\circ$ one has that 
\[
{\rm{Res}}_{\infty}\left(f\cdot g d\tau\right)=\sum_{\substack{0\leq m\leq M\\\mu\in A_\Lambda}}\alpha_{-m,\mu}c_{m,\mu}(g)=0,
\]
which via $\Psi_\Lambda\otimes \mathbb{C}$ descends to \eqref{sec2:eq:univ_relation} on ${\rm{Pic}}\left(\mathcal{F}_\Lambda\right)\otimes\mathbb{C}$.
\end{proof}

Theorems \ref{intro:thm:thm2} and \ref{intro:thm:thm3} now follow from Proposition \ref{sec2:prop:rels}. Using the formula in \cite{BK01}*{Theorem 7} one obtains the following:

\begin{corollary}
\label{sec2:thm:lambda}
Let $\Lambda$ be an even lattice of signature $(2,n)$ and fix $k={\rm{rk}}(\Lambda)/2$. Then the following relations holds in ${\rm{Pic}}_{\mathbb{Q}}\left(\mathcal{F}_\Lambda\right)$:
\begin{enumerate}
\item Let $N = \lfloor {\rm{rk}}(\Lambda)/24 \rfloor + 1$. Relation \eqref{sec2:eq:univ_relation} holds with $M=N$ and
\begin{equation*}
\alpha_{-m,\mu}=\sum_{\substack{n_1+n_2=-m\\\mu\in A_\Lambda\\ n_2\equiv q_\Lambda(\mu)\mod\mathbb{Z}}}P_{24N}(n_1+N)\cdot c_{n_2,\mu}(E),
\end{equation*}

where $P_M(n)$ counts the number of partitions of $n$ as a sum of $M$ partitions; 
\begin{equation}
\label{intro:eq:P}
P_M(n)=\sum_{k_1+\ldots+k_M=n}p(k_1)\cdots p(k_M),
\end{equation}
and $E$ is the vector-valued Eisenstein series defined in \eqref{sec2:eq:E} of weight $1+12N-n/2$ attached to any lattice $\widetilde{\Lambda}$ with discriminant group isomorphic to $\left(A_{\Lambda(-1)},q_{\Lambda(-1)}\right)=\left(A_\Lambda,-q_\Lambda\right)$.
\item Assume that there exists a positive-definite lattice $\widetilde{\Lambda}$ of rank $2-n+24N$ such that $\left(A_{\widetilde{\Lambda}},q_{\widetilde{\Lambda}}\right)\cong\left(A_\Lambda,-q_\Lambda\right)$. Then relation \eqref{sec2:eq:univ_relation} holds with $M=N$ and
\begin{equation*}
\alpha_{-m,\mu}=\sum_{\substack{n_1+n_2=-m\\\mu\in A_\Lambda\\ n_2\equiv q_\Lambda(\mu)\mod\mathbb{Z}}}P_{24N}(n_1+N)\cdot \theta_{n_2,\mu}(\widetilde{\Lambda}),
\end{equation*}
where $\theta_{n_2,\mu}$ is the number of elements in the class $\mu+\widetilde{\Lambda}$ of square $n_2$.
\end{enumerate}

\end{corollary}

\begin{proof}
Taking $f\in {\rm{M}}_{2-k,\Lambda(-1)}^!$ as the product
\[
f=\left(\Delta(\tau)\right)^{-N}\cdot E_{2-k+12N,\widetilde{\Lambda}}(\tau),
\] 
one obtains the first equation. Note that the fact that the coefficients $\alpha_{m,\mu}$ are rational is \cite{BK01}*{Corollary 8}. The second equality follows by taking 
\[
f=\left(\Delta(\tau)\right)^{-N}\cdot \Theta_{\widetilde{\Lambda}}(\tau),
\]
where $\Theta_{\widetilde{\Lambda}}$ is the vector-valued theta series attached to $\widetilde{\Lambda}$ of weight $\rm{rk}(\widetilde{\Lambda})/2$.
\end{proof}

\begin{remark}
    Write $\ell(A_\Lambda)$ for the minimal number of generators of the discriminant form $A_\Lambda$. Then, according to \cite{Nikulin}*{Corollary 1.10.2}, there is an even positive definite lattice $\widetilde{\Lambda}$ as in Corollary \ref{sec2:thm:lambda}, if $2 - n + 24N > \ell(A_\Lambda)$.
\end{remark}

Theorems \ref{intro:thm:main1}, \ref{intro:thm:thm2}, and \ref{intro:thm:thm3}  now follow from Proposition \ref{sec2:prop:generation} and Corollary \ref{sec2:thm:lambda}.

\begin{proof}[Proof of Theorem \ref{intro:thm:thm2} and \ref{intro:thm:thm3}]
Theorem \ref{intro:thm:thm2} follows from Corollary \ref{sec2:thm:lambda} by taking $N=1$, and $\widetilde{\Lambda}=U^{\oplus 3}\oplus\mathbb{Z}\ell^*$, and Theorem \ref{intro:thm:thm3} follows by taking $\widetilde{\Lambda}$ as the orthogonal complement in $E_8$ of an element of square $2d$.
\end{proof}

\begin{proof}[Proof of Theorem \ref{intro:thm:main1}]
By Proposition \ref{sec2:prop:generation} with $\Lambda$ the $2d$-polarized K3 lattice $U^{\oplus 2}\oplus E_8(-1)^{\oplus{2}}\oplus \mathbb{Z}\ell$, one obtains that ${\rm{Pic}}_{\mathbb{Q}}\left(\mathcal{F}_{2d}\right)$ is generated by $\lambda$ and $H_{m,\mu}$ with $m\leq 21/24$. By Corollary \ref{sec2:thm:lambda}, the Hodge class $\lambda$ can be expressed as a $\mathbb{Q}$-linear combination of $H_{m,\mu}$'s with $m\leq 1$. Generation by NL-divisors follows from \eqref{sec2:eq:H=P} and the established relation between NL divisors and primitive Heegner divisors.
\end{proof}

The same argument leads to the following more general statement that applies, for example, to  moduli spaces of (lattice) polarized hyperk\"{a}hler varieties.

\begin{theorem}
\label{sec2:thm:general}
Let $\Lambda$ be an even lattice of signature $(2,n)$ splitting off two copies of the hyperbolic planes. Then the Picard group ${\rm{Pic}}_{\mathbb{Q}}\left(\mathcal{F}_\Lambda\right)$ is generated by primitive Heegner divisors $P_{m,\mu}$ with $0<m\leq \lfloor {\rm{rk}}(\Lambda)/24\rfloor + 1$.
\end{theorem}

\begin{proof}
The theorem follows from Proposition \ref{sec2:prop:generation} and Corollary \ref{sec2:thm:lambda} with $N=\lfloor {\rm{rk}}(\Lambda)/24\rfloor + 1$ and $\widetilde{\Lambda}=\Lambda(-1)$.
\end{proof}

\section{Slope of divisors on the moduli space of cubic fourfolds}
\label{sec:cubics}
Let $\mathcal{M}_0$ be the moduli space of smooth cubic fourfolds and $\mathcal{M}$ the partial compactification by allowing at worst simple singularities (in the sense of Arnold), see \cite{Laz09}. The {\it{extended period map}} \cites{Voi86, Voi08, Loo09, Laz10} is given by an open immersion 
\begin{equation}
\label{sec4:eq:P}
\mathcal{P}:\mathcal{M}\longrightarrow \mathcal{D}\big/\Gamma,
\end{equation}
where $\mathcal{D}$ is the period domain for cubics, that is, the space of polarized Hodge structures of K3-type on the abstract lattice 
\begin{equation}
\label{sec4:eq:H4=L}
H^4(X,\mathbb{Z})_{\rm{prim}}\left(-1
\right)\cong \Lambda_{\rm{cubic}} : = U^{\oplus 2}\oplus E_{8}(-1)\oplus A_2(-1).
\end{equation}
The arithmetic group $\Gamma$ is the monodromy group that after the isomorphism \eqref{sec4:eq:H4=L} can be identified with $\Gamma_{\Lambda_{\rm{cubic}}}$, see \cites{Ebe84, Bea86}. We denote by $u,v$ the standard basis of $A_{2}(-1)$. After taking the closure of the image via $\mathcal{P}$, boundary $\mathcal{M}\setminus \mathcal{M}_0$ is sent to $\mathcal{C}_6=H_{1,0}$, and complement of the image of \eqref{sec2:eq:P} is given by $\mathcal{C}_{2}=H_{\frac{1}{3},\frac{2u+v}{3}}$. In this case ${\rm{Pic}}_{\mathbb{Q}}\left(\mathcal{D}\big/\Gamma\right)$ is two dimensional and via the isomorphism
\[ \left({\rm{Mod}}_{11,\Lambda_{\rm{cubic}}}^\circ\right)^\vee\cong{\rm{Pic}}_{\mathbb{Q}}\left(\mathcal{D}\big/\Gamma\right)
\]
identifying $H_{m,\mu}$ with the corresponding Fourier coefficient extraction functional, one obtains the following relation between $\mathcal{C}_2, \mathcal{C}_6$ and $\lambda$:
\begin{equation}
\label{sec4:eq:CCL}
\mathcal{C}_6=96\lambda-54\mathcal{C}_{2}.
\end{equation}
Further, by \cite{BBFW25} the cone generated by positive linear combinations of Hassett divisors (or NL divisors) is generated by $\mathcal{C}_2$ and $\mathcal{C}_6$:
\begin{equation}
\label{sec2:eq:NL-cone}
{\rm{Eff}}^{NL}\left(\mathcal{D}\big/\Gamma\right):=\left\langle \mathcal{C}_d\left|\;\; d\equiv 0,2\mod 6\right.\right\rangle_{\mathbb{Q}_{\geq 0}}=\left\{\alpha\mathcal{C}_2+\beta\mathcal{C}_6\left|\;\;\alpha,\beta\in\mathbb{Q}_{\geq 0}\right.\right\}.
\end{equation}
The divisor $\mathcal{C}_2$ gets contracted after a small map resolving the rational map $\mathcal{D}\big/\Gamma\dashrightarrow\overline{\mathcal{M}}$ to the GIT compactification of $\mathcal{M}$, see \cite{Laz10}*{Section 6}. In particular it is extremal, i.e. in the boundary of $\overline{\rm{Eff}}\left(\mathcal{D}\big/\Gamma\right)$.  

\begin{definition}
Let $E$ be an effective divisor on $\mathcal{M}$. We define the {\textit{slope of $E$}} as the quotient 
\[
s(E)=\begin{cases}\frac{\alpha}{\beta},\;\;\; \hbox{if }\alpha,\beta>0 \\
\infty,\;\;\;\hbox{otherwise.}
\end{cases}
\]
where $\overline{E}=\alpha \lambda-\beta \mathcal{C}_2$ and $\overline{E}$ is the closure of $E$ in $\mathcal{D}\big/\Gamma$. Similarly, we define the slope of $\mathcal{M}$ as
\[
s(\mathcal{M}):=\inf\left\{s(E)\left|\;\; E\hbox{ effective divisor in }\mathcal{M}\right.\right\}. 
\]
\end{definition}

Note that $s(\mathcal{M})\leq s(\mathcal{C}_6)=16/9$ and from \cite{BBFW25} (see \eqref{sec2:eq:NL-cone}) one has that any effective divisor $D$ supported on Hassett divisors satisfies $s(D)\geq 16/9$. 

Theorem \ref{intro:thm:s(M)} follows from the following observation.

\begin{lemma}
\label{sec4:lemma:gamma}
Fix $\varepsilon:\overline{\mathcal{D}\big/\Gamma}\longrightarrow\overline{\mathcal{D}\big/\Gamma}^{\rm{BB}}$ a normal $\mathbb{Q}$-factorial compactification of $\mathcal{D}\big/\Gamma$ over the Baily--Borel compactification and $\gamma\in N_1\left(\overline{\mathcal{D}\big/\Gamma}\right)_\mathbb{Q}$ the class of a curve moving over an algebraic base birationally covering an irreducible NL divisor $\mathcal{C}_d$ and intersecting trivially the divisoral part of the boundary of $\overline{\mathcal{D}\big/\Gamma}$. Then
\[
\frac{\overline{\mathcal{C}}_2\cdot \gamma}{\varepsilon^*\lambda\cdot\gamma}\leq s(\mathcal{M}),
\]
where $\overline{\mathcal{C}}_2$ stands for the closure of $\mathcal{C}_2$ in $\overline{\mathcal{D}\big/\Gamma}$.
\end{lemma}

\begin{proof}
Note first that if $D=D_1+D_2$ is an effective decomposition, then $s(D)\geq \min\{s(D_1),s(D_2)\}$ and to bound $s(\mathcal{M})$, we can assume $D$ is irreducible. Further, if $D$ is of NL-type, then by \cite{BBFW25} (see \eqref{sec2:eq:NL-cone}) one has that $s(D)\geq 16/9$. Assume $D\not\in{\rm{Eff}}^{NL}\left(\mathcal{D}\big/\Gamma\right)$. Then $\overline{D}\cdot\gamma\geq 0$ where $\overline{D}$ is the closure of $D$ in $\overline{\mathcal{D}\big/\Gamma}$. Indeed, if $\gamma\cdot\overline{D}<0$, since $\gamma$ is a covering curve for some $\overline{\mathcal{C}}_d$, then by our irreducibility assumption, after restricting to the interior $D$ and $\mathcal{C}_d$ would have to be proportional, but $D$ is assumed to be outside the NL-cone. If we write $\overline{D}=\alpha\varepsilon^*\lambda-\beta\overline{\mathcal{C}}_2+E$ with $E$ some divisor (not necessarily effective) supported on the boundary of $\overline{\mathcal{D}\big/\Gamma}$, then $\overline{D}\cdot\gamma\geq 0$ and $E\cdot\gamma=0$ implies the inequality.
\end{proof}

\begin{proof}[Proof of Theorem \ref{intro:thm:s(M)}]
Consider the finite map $\phi:\mathcal{F}_{6}\longrightarrow \mathcal{D}\big/\Gamma$ induced by the embedding $\Lambda_{6}\hookrightarrow \Lambda_{\rm{cubic}}$ sending $\ell$ to $u-v$. The image of the map $\phi$ is the divisor $\mathcal{C}_6$. Let $\overline{(\mathcal{D}\big/\Gamma)}^{tor}$ be a toroidal compactification of $(\mathcal{D}\big/\Gamma)$, and $\overline{\mathcal{F}}_6$ a normal $\mathbb{Q}$-factorial compactification of $\mathcal{F}_6$ resolving the rational map $\overline{\mathcal{F}}_6^{BB}\dasharrow \overline{(\mathcal{D}\big/\Gamma)}^{tor}$. The relevant diagram is the following:     
\begin{equation}
\label{sec4:eq:phi}
\begin{tikzcd}
\overline{\mathcal{F}}_6\arrow[d, "\delta"']\arrow[r, "\overline{\phi}"]&\overline{(\mathcal{D}\big/\Gamma)}^{tor}\arrow[d, "\varepsilon"]\\
\overline{\mathcal{F}}_6^{BB}\arrow[dashed, ur]\arrow[r, "\overline{\phi}^{BB}"']&\overline{(\mathcal{D}\big/\Gamma)}^{BB}.
\end{tikzcd}
\end{equation}
The Hodge class $\lambda$ is ample on $\overline{(\mathcal{D}\big/\Gamma)}^{BB}$. Consider the curve class $\gamma=\overline{\mathcal{C}}_6\cdot\left(\varepsilon^*\lambda\right)^{18}$. Since the pull-back $\left(\overline{\phi}^{BB}\right)^*\lambda$ is the Hodge class $\lambda'$ on $\overline{\mathcal{F}}_6^{BB}$, by the projection formula
\begin{equation}
\label{sec3:eq:proof_slope}
\frac{\overline{\mathcal{C}}_2\cdot \gamma}{\varepsilon^*\lambda\cdot\gamma}=\frac{\overline{\phi}^*\left(\overline{\mathcal{C}}_2\right)\cdot \left(\delta^*\lambda'\right)^{18}}{\left(\delta^*\lambda'\right)^{19}}=\frac{{\rm{deg}}\left(\overline{\phi}^*\left(\overline{\mathcal{C}}_2\right)\right)}{{\rm{vol}}\left(\overline{\mathcal{F}}_6\right)},
\end{equation}
where ${\rm{deg}}(\cdot)$ stands for the Baily--Borel degree of the divisor $\overline{\phi}^*\left(\overline{\mathcal{C}}_2\right)$, and ${\rm{vol}}\left(\overline{\mathcal{F}}_6\right)$ the volume of the orthogonal modular variety $\mathcal{F}_6$ with respect to the K\"{a}hler form $\omega=c_1(\lambda')$. Further, from \cite{Kud97}*{Section 9} (see also \cite{BFZ25}*{Proposition 2.2}) one obtains
\[
\overline{\phi}^*(\overline{\mathcal{C}}_2)=H_{\frac{1}{3},2\ell_*}.
\]
Finally, from \cite{Kud03}*{Theorem I} the ratio \eqref{sec3:eq:proof_slope} is given by minus the $q^{1/3}$-coefficient of the $\frak{e}_{2\ell_*}$-component of the vector-valued Eisenstein series defined in \eqref{sec2:eq:E} of weight $21/2$ attached to $\Lambda_6$. This can be computed using \cite{weilrep}, see also \cite{Wil18}:
\[
\left(E_{\frac{21}{2}, \Lambda_6}\right)_{\frak{e}_{2\ell_*}}=-\frac{523777}{206215591}q^{1/3} - \frac{274609995265}{206215591}q^{4/3} - \frac{55921251768096}{206215591}q^{7/3} +\ldots
\]

\end{proof}

\subsection{K3 surfaces of degree $2$} The picture for the moduli space $\mathcal{F}_2$ parameterizing K3 surfaces of degree $2$ is analogous to the one of cubic fourfolds, see \cite{Sha80}. Indeed, if one sees K3 surfaces of degree $2$ as sextic double planes with at worst isolated rational double points, then the moduli space $\mathcal{N}$ of such sextic double planes admits an embedding
\[
\mathcal{N}\longrightarrow D_{\Lambda_2}\big/\Gamma_{\Lambda_2}=\mathcal{F}_2,
\]
with boundary in $\mathcal{F}_2$ given exactly by the unigonal divisor $D_{1,1}$, see \cite{Sha80}*{Corollary 6.2}. This divisor is extremal in the effective cone ${\rm{Eff}}\left(\mathcal{F}_2\right)$ as it gets contracted to a point via the rational map to the GIT moduli space of plane sextics. 

\begin{definition}
Let $E$ be an effective divisor on $\mathcal{F}_2$. We define the {\textit{slope of $E$}} as the quotient 
\[
s(E)=\begin{cases}\frac{\alpha}{\beta},\;\;\; \hbox{if }\alpha,\beta>0 \\
\infty,\;\;\;\hbox{otherwise.}
\end{cases}
\]
where $E=\alpha \lambda-\beta \mathcal{D}_{1,1}$. Similarly, we define the slope of $\mathcal{F}_2$ as
\[
s(\mathcal{F}_2):=\inf\left\{s(E)\left|\;\; E\hbox{ effective divisor in }\mathcal{F}_2\right.\right\}. 
\]
\end{definition}

\begin{proposition}
\label{sec4:prop:K3_2}
The slope of the moduli space of K3 surfaces of degree $2$ satisfies the following bound:
\[
\frac{1}{1984}\leq s(\mathcal{F}_2)\leq \frac{150}{57}
\]
\end{proposition}

The proof is analogous to that of Theorem \ref{intro:thm:s(M)}.

\begin{proof}[Proof of Proposition \ref{sec4:prop:K3_2}]
Let $\Lambda_2$ be the lattice defined in \eqref{sec2:eq:K3lattice}, and $L=U\oplus \mathbb{Z}(e+f)\oplus E_8(-1)^{\oplus 2}\oplus \mathbb{Z}\ell$ the orthogonal complement of $e-f\in\Lambda_2$, where $\{e,f\}$ are the standard generators of one copy of $U$. The induced map
\[
\phi:\mathcal{D}_L\big/\Gamma_L\longrightarrow \mathcal{F}_2
\]
is finite onto the irreducible divisor $P_{1,0}$ and $s(P_{1,0})=\frac{150}{57}$. We keep the same notation as in \eqref{sec4:eq:phi}. By \cite{BBFW25}, any effective divisor $E$ supported on NL divisors has slope $s(E)\geq s(P_{1,0})$. In particular, the same proof as Lemma \ref{sec4:lemma:gamma} leads to the inequality
\[
\frac{\overline{D}_{1,1}\cdot \gamma}{\varepsilon^*\lambda\cdot\gamma}\leq s\left(\mathcal{F}_2\right)
\]
for any curve $\gamma$ covering an NL divisor. Taking $\gamma=\overline{\mathcal{P}}_{1,0}\cdot\left(\varepsilon^*\lambda\right)^{17}$, by the projection formula and \cite{BFZ25}*{Proposition 2.2} one obtains
\[
\frac{\overline{D}_{1,1}\cdot \gamma}{\varepsilon^*\lambda\cdot\gamma}=\frac{\overline{\phi}^*\left(\overline{D}_{1,1}\right)\cdot \left(\delta^*\lambda'\right)^{17}}{\left(\delta^*\lambda'\right)^{18}}=\frac{{\rm{deg}}\left(H_{\frac{1}{4},\ell_*}^L\right)}{{\rm{vol}}\left(\mathcal{D}_L\big/\Gamma_L\right)}\leq s\left(\mathcal{F}_2\right).
\]
Finally, by \cite{Kud03}*{Theorem I} and using \cites{weilrep, Wil18}, one obtains the desired lower bound.
\end{proof}

\bibliography{Bibliography}
\bibliographystyle{alpha}
\end{document}